\documentclass[12pt]{amsart}

\usepackage[colorlinks=true, pdfstartview=FitV, linkcolor=blue, citecolor=blue, urlcolor=blue, breaklinks=true]{hyperref}
\usepackage{
  amsmath,
  amsfonts,
  amssymb,
  amsthm,
  amscd,
  comment,
  paralist,
  etoolbox,
  mathtools,
  enumitem,
  mathdots
}
\usepackage[usenames,dvipsnames]{xcolor}
\usepackage{mathrsfs} 

\usepackage[all]{xy}
\usepackage{tikz}
\usetikzlibrary{arrows}

\makeatletter
\def\namedlabel#1#2{\begingroup
    #2%
    \def\@currentlabel{#2}%
    \phantomsection\label{#1}\endgroup
}

\newcommand\Z{\mathbb{Z}}
\newcommand\Q{\mathbb{Q}}

\newcommand\F{\mathbb{F}}

\newcommand{\md}{\textup{-mod}}

\newcommand\pr[1]{\prescript{r}{}{#1}}
\newcommand\pl[1]{\prescript{\ell}{}{#1}}

\DeclareMathOperator{\HOM}{HOM}

\DeclareMathOperator{\tr}{tr}
\DeclareMathOperator{\id}{id}

\newtheorem{theo}{Theorem}[section]
\newtheorem{prop}[theo]{Proposition}
\newtheorem{lem}[theo]{Lemma}
\newtheorem{cor}[theo]{Corollary}

\theoremstyle{definition}
\newtheorem{defin}[theo]{Definition}
\newtheorem{rem}[theo]{Remark}
\newtheorem{eg}[theo]{Example}

\numberwithin{equation}{section}
\allowdisplaybreaks

\newtoggle{comments}
\newtoggle{details}
\newtoggle{prelimnote}
\newtoggle{detailsnote}

\toggletrue{detailsnote}   

\iftoggle{comments}{%
  \newcommand{\acomments}[1]{
    \ \\
    {\color{red}
      \textbf{AS:} #1
    }
    \ \\
    }
  \newcommand{\mcomments}[1]{
    \ \\
    {\color{red}
      \textbf{MM:} #1
    }
    \ \\
    }
}{%
  \newcommand{\acomments}[1]{}
  \newcommand{\mcomments}[1]{}
}

\iftoggle{details}{%
  \newcommand{\details}[1]{
      \ \\
      {\color{OliveGreen}
        \textbf{Details:} #1
      }
      \\
  }
}{%
  \newcommand{\details}[1]{}
}

\iftoggle{prelimnote}{%
  \newcommand{\prelim}{\textsc{Preliminary version} \bigskip}
}{%
  \newcommand{\prelim}{}
}

\title{Nested Frobenius extensions of graded superrings}

\author{Edward Poon}
\address{E.~Poon: Department of Mathematics and Statistics, University of Ottawa}
\email{epoon061@uottawa.ca}

\author{Alistair Savage}
\address{A.~Savage: Department of Mathematics and Statistics, University of Ottawa}
\urladdr{\url{http://alistairsavage.ca}}
\email{alistair.savage@uottawa.ca}

\begin{document}

\begin{abstract}
  We prove a nesting phenomenon for twisted Frobenius extensions.  Namely, suppose $R \subseteq B \subseteq A$ are graded superrings such that $A$ and $B$ are both twisted Frobenius extensions of $R$, $R$ is contained in the center of $A$, and $A$ is projective over $B$.  Our main result is that, under these assumptions, $A$ is a twisted Frobenius extension of $B$.  This generalizes a result of Pike and the second author, which considered the case where $R$ is a field.
\end{abstract}


\prelim

\maketitle
\thispagestyle{empty}

\tableofcontents

%
\section{Introduction}
%

Frobenius extensions, which are a natural generalization of Frobenius algebras, appear frequently in many areas of mathematics, from topological quantum field theory to categorical representation theory.  Several generalizations of Frobenius extensions have been introduced since their inception.  In particular, Nakayama and Tsuzuku introduced Frobenius extensions of the second kind in \cite{NT60}.  These were further generalized to the concept of $(\alpha,\beta)$-Frobenius extensions in \cite{Mor65}, where $\alpha$ and $\beta$ are automorphisms of the rings involved.  The corresponding theory for graded superrings was then developed in \cite{PS16}, where they were called \emph{twisted Frobenius extensions}.

In the literature, one finds that many examples of (twisted) Frobenius extensions arise from certain types of subobjects.  For instance, if $H$ is a finite index subgroup of $G$, then the group ring $R[G]$ is a Frobenius extension of $R[H]$, where $R$ is a commutative base ring.  This example dates back to the original paper \cite{Kas54} on Frobenius extensions.  Another example comes from the theory of Hopf algebras.  In particular, it was shown in \cite[Cor.~3.6(1)]{Sch92} that if $K$ is a Hopf subalgebra of $H$, then $H$ is a Frobenius extension of $K$ of the second kind.  Yet another example comes from Frobenius algebras themselves.  Namely, it was shown (in the more general graded super setting) in \cite[Cor.~7.4]{PS16} that if $A$ is a Frobenius algebra over a field, $B$ is a subalgebra of $A$ that is also a Frobenius algebra, and $A$ is projective over $B$, then $A$ is a twisted Frobenius extension of $B$.

The goal of the current paper is to shed more light on this ``nesting'' phenomenon.  Namely, we consider the situation where we have graded superrings $R \subseteq B \subseteq A$, such that $A$ and $B$ are both twisted Frobenius extensions of $R$, and $R$ is contained in the center of $A$.  We call these \emph{nested Frobenius extensions}.  We first prove that these assumptions imply $A$ and $B$ are \emph{untwisted} Frobenius extensions of $R$ (see Corollary~\ref{cor:R-central-implies-untwisted}).  Then, our main result (Theorem~\ref{theo:main-theorem}) is that, provided $A$ is projective over $B$, it follows that $A$ is a twisted Frobenius extension of $B$.  The twisting is given in terms of the Nakayama automorphisms of $A$ and $B$.  In particular, even though $A$ and $B$ are untwisted Frobenius extensions of $R$, $A$ can be a nontrivially twisted Frobenius extension of $B$.  This result can be viewed as a generalization of \cite[Cor.~7.4]{PS16} to the setting of arbitrary supercommutative ground rings.

The organization of the paper is as follows.  We begin in Section~\ref{sec:background} by recalling the definition of twisted Frobenius extensions of graded superrings, together with some related results.   In Section~\ref{sec:nestings}, we examine nested Frobenius extensions $R \subseteq B \subseteq A$ where $R$ is contained in the center of $A$.  We begin by proving that $A$ and $B$ are, in fact, \emph{untwisted} Frobenius extensions of $R$ (Corollary~\ref{cor:R-central-implies-untwisted}).  Then, after establishing several important lemmas, we prove our main result (Theorem~\ref{theo:main-theorem}), that $A$ is a twisted Frobenius extension of $B$, provided $A$ is projective over $B$.  We conclude in Section~\ref{sec:applications} with several applications of our main result.  In particular, we explain how the aforementioned examples of group rings and Hopf algebras can be deduced from our main theorem.  We also give an example arising from nilcoxeter rings.

\subsection*{Acknowledgements}

This paper is the result of a research project completed in the context of an Undergraduate Student Research Award from the Natural Sciences and Engineering Research Council of Canada (NSERC), received by the first author.  The second author was supported by an NSERC Discovery Grant.

\iftoggle{detailsnote}{
  \subsection*{Note on the arXiv version} For the interested reader, the tex file of the arXiv version of this paper includes hidden details of some straightforward computations and arguments that are omitted in the pdf.  These details can be displayed by switching the \texttt{details} toggle to true in the tex file and recompiling.
}{}

%
\section{Twisted Frobenius extensions} \label{sec:background}
%

In this section we recall the definition of twisted Frobenius extensions, together with some of their properties that will be used in this paper.  We refer the reader to \cite{PS16} for further details.

We fix an abelian group $\Lambda$ and by \emph{graded}, we mean $\Lambda$-graded.  In particular, a graded superring is a $\Lambda \times \Z_2$-graded ring.  In other words, if $A$ is a graded superring, then
\[
  A = \bigoplus_{\lambda \in \Lambda,\, \pi \in \Z_2} A_{\lambda,\pi},\quad A_{\lambda,\pi} A_{\lambda',\pi'} \subseteq A_{\lambda+\lambda',\pi+\pi'},\quad \lambda,\lambda' \in \Lambda,\, \pi, \pi' \in \Z_2.
\]
We denote the multiplicative unit of $A$ by $1_A$.  To avoid repeated use of the modifiers ``graded'' and ``super'', from now on we will use the term \emph{ring} to mean graded superring and \emph{subring} to mean graded sub-superring.  Similarly, by an automorphism of a ring, we mean an automorphism as graded superrings (homogeneous of degree zero).

We will use the term \emph{module} to mean graded supermodule.  In particular, a left $A$-module $M$ is a $\Lambda \times \Z_2$-graded abelian group with a left $A$-action such that
\[
  A_{\lambda,\pi} M_{\lambda',\pi'} \subseteq M_{\lambda+\lambda',\pi+\pi'}, \quad \lambda,\lambda' \in \Lambda,\ \pi,\pi' \in \Z_2,
\]
and similarly for right modules.  If $v$ is a homogeneous element in a ring or module, we will denote by $|v|$ its $\Lambda$-degree and by $\bar v$ its $\Z_2$-degree.  Whenever we write an expression involving degrees of elements, we will implicitly assume that such elements are homogeneous.

For $M$, $N$ two $\Lambda \times \Z_2$-graded abelian groups, we define a $\Lambda \times \Z_2$-grading on the space $\HOM_\Z(M,N)$ of all $\Z$-linear maps by setting $\HOM_\Z(M,N)_{\lambda,\pi}$, $\lambda \in \Lambda$, $\pi \in \Z_2$, to be the subspace of all homogeneous maps of degree $(\lambda,\pi)$. That is,
\[
  \HOM_\Z(M,N)_{\lambda,\pi} = \{f \in \HOM_\Z(M,N) \mid  f(M_{\lambda',\pi'}) \subseteq N_{\lambda+\lambda',\pi+\pi'} \ \forall\ \lambda' \in \Lambda,\ \pi' \in \Z_2\}.
\]

For $A$-modules $M$ and $N$, we define the $\Lambda \times \Z_2$-graded abelian group
\[
  \HOM_A(M,N) = \bigoplus_{\lambda \in \Lambda,\, \pi \in \Z_2} \HOM_A(M,N)_{\lambda,\pi},
\]
where the homogeneous components are defined by
\[
  \HOM_A(M,N)_{\lambda,\pi} = \{f \in \HOM_\Z(M,N)_{\lambda,\pi} \mid f(am) = (-1)^{\pi \bar a} a f(m) \ \forall\ a \in A,\ m \in M\}.
\]
We let $A\md$ denote the category of left $A$-modules, with set of morphisms from $M$ to $N$ given by $\HOM_A(M,N)_{0,0}$.  Similarly, we have the category of right $A$-modules with morphisms from $M$ to $N$ given by
\[
  \{ f \in \HOM_\Z(M,N)_{0,0} \mid f(ma) = f(m)a \ \forall\ m \in M,\ a \in A\}.
\]
We will call elements of $\HOM_A(M,N)_{\lambda,\pi}$ \emph{homomorphisms of degree $(\lambda,\pi)$} and, if they are invertible, \emph{isomorphisms of degree $(\lambda,\pi)$}.  Note that they are not morphisms in the category $A\md$ unless they are of degree $(0,0)$.  We use similar terminology for right modules.

If $M$ is a left $A$-module, we let $\pl{a}$ denote the operator given by the left action by $a$, that is,
\begin{equation} \label{eq:left-mult-action}
  \pl{a}(m) = am ,\quad a \in A,\ m \in M.
\end{equation}
If $M$ is a right $A$-module, then for each homogeneous $a \in A$, we define a $\Z$-linear operator
\begin{equation} \label{eq:sign-right-action}
  \pr{a} \colon M\to M,\quad \pr{a}(m) = (-1)^{\bar a \bar m} ma, \quad a \in A,\ m \in M.
\end{equation}

If $A_1$ and $A_2$ are rings, then, by definition, an $(A_1,A_2)$-bimodule $M$ is both a left $A_1$-module and a right $A_2$-module, such that the left and right actions commute:
\[
  (a_1m)a_2 = a_1(ma_2) \quad \text{for all } a_1 \in A_1,\ a_2 \in A_2,\ m \in M.
\]
If $M$ is an $(A_1,A_2)$-bimodule and $N$ is a left $A_1$-module, then $\HOM_{A_1}(M,N)$ is a left $A_2$-module via the action
\begin{equation} \label{eq:left-hom-action}
  a \cdot f = (-1)^{\bar a \bar f} f \circ \pr{a},\quad a \in A_2,\ f \in \HOM_{A_1}(M,N),
\end{equation}
and $\HOM_{A_1}(N,M)$ is a right $A_2$-module via the action
\begin{equation} \label{eq:right-hom-action}
  f \cdot a = (-1)^{\bar a \bar f} (\pr{a}) \circ f,\quad a \in A_2,\ f \in \HOM_{A_1}(N,M).
\end{equation}

For $\lambda \in \Lambda$, $\pi \in \Z_2$, and an $A$-module $M$, we let $\{\lambda,\pi\}M$ denote the $\Lambda \times \Z_2$-graded abelian group that has the same underlying abelian group as $M$, but a new grading given by $(\{\lambda,\pi\}M)_{\lambda',\pi'}=M_{\lambda'-\lambda,\pi'-\pi}$.  Abusing notation, we will also sometimes use $\{\lambda,\pi\}$ to denote the map $M \to \{\lambda,\pi\}M$ that is the identity on elements of $M$.  We define a left action of $A$ on $\{\lambda,\pi\}M$ by $a \cdot \{\lambda,\pi\} m= (-1)^{\pi \bar a} \{\lambda,\pi\}am$.  In this way, $\{\lambda,\pi\}$ defines a functor from the category of $A$-modules to itself that leaves morphisms unchanged.

Suppose $M$ is a left $A$-module, $N$ is a right $A$-module, and $\alpha$ is a ring automorphism of $A$.  Then we can define the twisted left $A$-module ${}^\alpha M$ and twisted right $A$-module $N^\alpha$ to be equal to $M$ and $N$, respectively, as graded abelian groups, but with actions given by
\begin{gather}
  \label{eq:left-twist} a \cdot m = \alpha(a) m,\quad a \in A,\ m \in {}^\alpha{M}, \\
  \label{eq:right-twist} n \cdot a = n \alpha(a),\quad a \in A,\ n \in N^\alpha,
\end{gather}
where juxtaposition denotes the original action of $A$ on $M$ and $N$.  If $\alpha$ is a ring automorphism of $A$, and $B$ is a subring of $A$, then we will also use the notation ${}^\alpha_B A_A$ to denote the $(B,A)$-bimodule equal to $A$ as a graded abelian group, with right action given by multiplication, and with left action given by $b \cdot a = \alpha(b)a$ (where here juxtaposition is multiplication in the ring $A$), even though $\alpha$ is not necessarily a ring automorphism of $B$.  We use ${}_A A^\alpha_B$ for the obvious right analogue.  By convention, when we consider twisted modules as above, operators such as $\pr{a}$ and $\pl{a}$ defined in \eqref{eq:left-mult-action} and \eqref{eq:sign-right-action} involve the right and left action (respectively) in the \emph{original} (i.e.\ untwisted) module.

\begin{defin}[Twisted Frobenius extension]
  Suppose $B$ is a subring of a ring $A$, that $\alpha$ is a ring automorphism of $A$, and that $\beta$ is a ring automorphism of $B$.  Furthermore, suppose $\lambda \in \Lambda$ and $\pi \in \Z_2$.  We call $A$ an \emph{$(\alpha,\beta)$-Frobenius extension} of $B$ of \emph{degree} $(-\lambda,\pi)$ if $A$ is finitely generated and projective as a left $B$-module, and there is a morphism of $(B,B)$-bimodules
  \[
    \tr \colon {}^\beta_B A_B^\alpha \to \{\lambda,\pi\} {}_B B_B
  \]
  satisfying the following two conditions:
  \begin{description}[style=multiline, labelwidth=0.7cm]
    \item[\namedlabel{T1}{T1}] If $\tr(Aa) = 0$ for some $a \in A$, then $a=0$.
    \item[\namedlabel{T2}{T2}] For every $\varphi \in \HOM_B \left( {}^\beta_B A, \left\{ \lambda, \pi \right\} {}_B B\right)$, there exists an $a \in A$ such that $\varphi = \tr \circ \pr{a}$.
  \end{description}
  The map $\tr$ is called a \emph{trace map}.  We will often view it as a map ${}^\beta_B A_A^\alpha \to {}_B B_B$ that is homogeneous of degree $(-\lambda,\pi)$.  If $A$ is an $(\alpha,\beta)$-Frobenius extension of $B$ for some $\alpha$ and $\beta$, we say that $A$ is a \emph{twisted Frobenius extension} of $B$.  If $A$ is an $(\id_A,\id_B)$-Frobenius extension of $B$, we call it a \emph{Frobenius extension} or \emph{untwisted Frobenius extension} (when we wish to emphasize that the twistings are trivial).
\end{defin}

\begin{rem}
  We say the extension is of degree $(-\lambda,\pi)$ since that is the degree of the trace map.  If $A$ and $B$ are concentrated in degree $(0,0)$, then $(\alpha,\beta)$-Frobenius extensions were defined in \cite[p.~41]{Mor65}.  In particular, an $(\id_A,\beta)$-Frobenius extension is sometimes called a \emph{$\beta^{-1}$-extension}, or a \emph{Frobenius extension of the second kind}; see \cite{NT60}.
\end{rem}

If $B$ is a subring of a ring $A$, then we define the \emph{centralizer} of $B$ in $A$ to be the subring of $A$ given by
\begin{equation}
  C_A(B) = \{a \in A \mid ab = (-1)^{\bar a \bar b} ba \text{ for all } b \in B\}.
\end{equation}

If $A$ is an $(\alpha,\beta)$-Frobenius extension of $B$, then we have the associated \emph{Nakayama isomorphism} (an isomorphism of rings)
\[
  \psi \colon C_A(B) \to C_A(\alpha(B)).
\]
which is the unique map satisfying
\begin{equation}
  \tr(c a) = (-1)^{\bar a \bar c} \tr \left( a \psi(c) \right)
  \quad \text{for all } a \in A,\ c \in C_A(B).
\end{equation}

\begin{prop} \label{prop:dual-sets-of-generators}
  The ring $B$ is an untwisted Frobenius extension of $R$ of degree $(-\lambda,\pi)$ if and only if there exists a homomorphism of $(R,R)$-bimodules $\tr \colon B \to R$ of degree $(-\lambda,\pi)$, and finite subsets $\{x_1,\dotsc, x_n\}$, $\{y_1,\dotsc,y_n\}$ of $B$ such that $(|y_i|, \bar y_i) = (\lambda - |x_i|, \pi - \bar x_i)$ for $i=1,\dotsc,n$, and
  \begin{equation} \label{eq:dual-bases}
    b = (-1)^{\pi \bar b} \sum_{i=1}^n (-1)^{\pi \bar x_i} \tr(by_i) x_i
    = \sum_{i=1}^n y_i \tr(x_i b) \quad \text{for all } b \in B.
  \end{equation}
  We call the sets $\{x_1,\dotsc,x_n\}$ and $\{y_1,\dotsc,y_n\}$ \emph{dual sets of generators} of $B$ over $R$.
\end{prop}

\begin{proof}
  This is a special case of \cite[Prop.~4.9]{PS16}, where the twistings are trivial.
\end{proof}

%
\section{Nested Frobenius extensions} \label{sec:nestings}
%

In this section, we introduce our main object of study, nested Frobenius extensions, and prove our main result (Theorem~\ref{theo:main-theorem}).  We begin with a simplification result.

\begin{lem} \label{lem:AR-identities}
  Suppose $A$ is an $(\alpha,\beta)$-Frobenius extension of $R$ of degree $(-\lambda,\pi)$, with trace map $\tr$ and Nakayama isomorphism $\psi$.  Furthermore, suppose $C_A(R) = A$.  Then $\alpha|_R = \beta$ and $\psi|_R = \id_R$.
\end{lem}

\begin{proof}
  For all $r \in R$ and $a \in \prescript{\beta}{R}A^\alpha_R$, we have
  \begin{multline*}
    \tr(ar)
    = \tr(a) \alpha^{-1}(r)
    = (-1)^{\bar r (\pi + \bar a)} \alpha^{-1}(r) \tr(a) \\
    = (-1)^{\bar r \bar a} \tr \left( \beta(\alpha^{-1}(r)) a \right)
    = \tr \left( a \beta(\alpha^{-1}(r)) \right),
  \end{multline*}
  where the second and fourth equalities follow from the fact that $C_A(R)=A$.  Since the trace map is linear, this implies
  \[
    \tr \left( a \left( r - \beta(\alpha^{-1}(r)) \right) \right) = 0
    \quad \text{for all } a \in  {}^\beta_R A^\alpha_R.
  \]
  By \ref{T1}, we have $\beta(\alpha^{-1}(r)) = r$ for all $r \in R$.  It follows that $\alpha|_R = \beta$.

  Similarly, for all $r \in R$ and $a \in \prescript{\beta}{R}A^\alpha_R$, we have
  \[
    \tr(ar)
    = (-1)^{\bar r \bar a} \tr(ra)
    = \tr(a \psi(r)),
  \]
  and so $\psi|_R = \id_R$ by \ref{T1}.
\end{proof}

\begin{cor} \label{cor:R-central-implies-untwisted}
  If $A$ is a twisted Frobenius extension of $R$ and $C_A(R) = A$, then $A$ is in fact an \emph{untwisted} Frobenius extension of $R$ of the same degree.
\end{cor}

\begin{proof}
  Suppose $A$ is an $(\alpha,\beta)$-Frobenius extension of $R$ of degree $(-\lambda,\pi)$, with trace map $\tr$ and Nakayama isomorphism $\psi$.  Furthermore, suppose that $C_A(R) = A$.  Then, by Lemma~\ref{lem:AR-identities}, $A$ is an $(\alpha,\alpha)$-Frobenius extension of $R$ and $\alpha(R) = \beta(R) = R$.  The result then follows immediately from \cite[Cor.~3.6]{PS16}.
\end{proof}

For the remainder of this paper, we fix rings
\[
  R \subseteq B \subseteq A, \quad \text{with } C_A(R) = A.
\]
This implies $C_B(R) = B$ and that $C_R(R) = R$.  In particular, $R$ is supercommutative, and so we do not distinguish between left and right $R$-modules.  In light of Corollary~\ref{cor:R-central-implies-untwisted}, we suppose that $A$ and $B$ are untwisted Frobenius extensions of $R$ of degrees $(-\lambda_A,\pi_A)$ and $(-\lambda_B,\pi_B)$, respectively.  We denote their trace maps by $\tr_A$ and $\tr_B$ and their Nakayama isomorphisms by $\psi_A$ and $\psi_B$, respectively.  We call $A$ and $B$ \emph{nested Frobenius extensions} of $R$.

\begin{rem}
  The assumption that $C_A(R)=A$ implies that $\psi_A$ and $\psi_B$ are ring automorphisms of $A$ and $B$, respectively.  In fact, this is precisely why we assume $C_A(R)=A$.
\end{rem}

For an $R$-module $M$, we define
\[
  M^\vee = \HOM_R(M,R).
\]
If, in addition, $M$ is a $(B,A)$-bimodule, then it is straightforward to verify that $M^\vee$ is an $(A,B)$-bimodule with action given by
\begin{equation} \label{eq:dual-action}
  a \cdot f \cdot b
  = (-1)^{\bar a \bar f} f \circ \pr{a} \circ \pl{b}
  = (-1)^{\bar a \bar f + \bar a \bar b} f \circ \pl{b} \circ \pr{a},
  \quad a \in A,\ b \in B,\ f \in M^\vee.
\end{equation}

\details{
  For $a \in A$, $f \in M^\vee$, $m,n \in M$, and $r \in R$, we have
  \begin{multline*}
    (a \cdot f)(rm+n)
    = (-1)^{\bar a \bar f} f \circ \pr{a}(rm+n)
    = (-1)^{\bar a \bar f + \bar a (\bar r + \bar m)} f(rma) + (-1)^{\bar a \bar f + \bar a \bar n} f(na)
    \\
    = (-1)^{\bar a \bar f + \bar a (\bar r + \bar m) + \bar r \bar f} rf(ma) + (-1)^{\bar a \bar f + \bar a \bar n} f(na)
    \\
    = (-1)^{\bar a \bar f + \bar a \bar r + \bar r \bar f} rf \circ \pr{a}(m) + (-1)^{\bar a \bar f} f \circ \pr{a}(n)
    = (-1)^{\bar r (\bar a + \bar f)} r(a \cdot f)(m) + (a \cdot f)(n),
  \end{multline*}
  and so $a \cdot f \in M^\vee$.

  For $a, a_1, a_2 \in A$ and $f,g \in M^\vee$ such that $\bar f = \bar g$, we have
  \begin{multline*}
    a_1 \cdot (a_2 \cdot f)
    = (-1)^{\bar a_2 \bar f} a_1 \cdot (f \circ \pr{a_2})
    = (-1)^{\bar a_1 (\bar a_2 + \bar f) + \bar a_2 \bar f} f \circ \pr{a_2} \circ \pr{a_1}
    \\
    = (-1)^{(\bar a_1 + \bar a_2) \bar f} f \circ \pr{(a_1 a_2)}
    = (a_1 a_2) \cdot f,
  \end{multline*}
  and
  \begin{gather*}
    a \cdot (f + g)
    = (-1)^{\bar a \bar f} (f+g) \circ \pr{a}
    = (-1)^{\bar a \bar f} f \circ \pr{a} + (-1)^{\bar a \bar g} \circ \pr{a}
    = a \cdot f + a \cdot g,
    \\
    1_A \cdot f = f \circ \pr{1_A} = f.
  \end{gather*}
  If, in addition, $\bar a_1 = \bar a_2$, then
  \[
    (a_1 + a_2) \cdot f
    = (-1)^{\bar a_1 \bar f} f \circ \pr{(a_1+a_2)}
    = (-1)^{\bar a_1 \bar f} f \circ \pr{a_1} + (-1)^{\bar a_2 \bar f} f \circ \pr{a_2}
    = a_1 \cdot f + a_2 \cdot f.
  \]
  Thus $M^\vee$ is indeed a left $A$-module under the action \eqref{eq:dual-action}.

  On the other hand, for $b \in B$, $f \in M^\vee$, $m,n \in M$, and $r \in R$, we have
  \begin{multline*}
    (f \cdot b)(rm + n)
    = f \circ \pl{b} (rm+n)
    = f (brm + bn)
    = (-1)^{\bar r (\bar f + \bar b)} rf(bm) + f(bn)
    \\
    = (-1)^{\bar r (\bar f + \bar b)} rf \circ \pl{b}(m) + f \circ \pl{b}(n)
    = (-1)^{\bar r (\bar f + \bar b)} r(f \cdot b)(m) + (f \cdot b)(n).
  \end{multline*}
  Thus, $f \cdot b \in M^\vee$.

  For $b, b_1, b_2 \in B$, $f, g \in B$, we have
  \begin{gather*}
    (f + g) \cdot b
    = (f + g) \circ \pl{b}
    = f \circ \pl{b} + g \circ \pl{b}
    = f \cdot b + g \cdot b,
    \\
    f \cdot (b_1 + b_2)
    = f \circ \pl{(b_1 + b_2)}
    = f \circ \pl{b_1} + f \circ \pl{b_2}
    = f \cdot b_1 + f \cdot b_2,
    \\
    (f \cdot b_1) \cdot b_2
    = \left( f \circ \pl{b_1} \right) \cdot b_2
    = \left( f \circ \pl{b_1} \right) \circ \pl{b_2}
    = f \circ \pl{(b_1b_2)}
    = f \cdot (b_1 b_2),
    \\
    f \cdot 1_B
    = f \circ \pl{1_B}
    = f.
  \end{gather*}
  Thus $M^\vee$ is indeed a right $B$-module under the action \eqref{eq:dual-action}.

  Finally, to see that the left and right action commute, note that, for $a \in A$, $b \in B$, $f \in M^\vee$, and $m \in M$, we have
  \begin{multline*}
    ((a \cdot f) \cdot b)(m)
    = (-1)^{\bar{a}\bar{f}}((f \circ \pr{a}) \cdot b)(m)
    = (-1)^{\bar{a}\bar{f}}(f \circ \pr{a} \circ \pl{b})(m)
    = (-1)^{\bar{a}\bar{f}}(f \circ \pr{a})(bm)
    \\
    = (-1)^{\bar{a}\bar{f} + \bar{a}(\bar{b} + \bar{m})} f(bma)
    = (-1)^{\bar{a}(\bar{f} + \bar{b}) + \bar{a}\bar{m}} f \circ \pl{b}(ma)
    = (-1)^{\bar{a}(\bar{f} + \bar{b})} (f \cdot b) \circ \pr{a}(m)
    = (a \cdot (f \cdot b))(m).
  \end{multline*}
}

Note that $B$ is naturally a $(B,B)$-bimodule via left and right multiplication.  We denote this bimodule by ${}_BB_B$ to emphasize the actions.  Therefore, if $M$ is a $(B,A)$-bimodule, then $\HOM_B(M,{}_BB_B)$ is an $(A,B)$-bimodule via the actions \eqref{eq:left-hom-action} and \eqref{eq:right-hom-action}.

\begin{lem}
  For any $(B,A)$-bimodule $M$, the map
  \begin{equation} \label{eq:dual-isom}
    \HOM_B \left( {}^{\psi_B}M, {}_BB_B \right) \to M^\vee,\quad f \mapsto \tr_B \circ f,
  \end{equation}
  is a homomorphism of $(A,B)$-bimodules of degree $(-\lambda_B,\pi_B)$.
\end{lem}

\begin{proof}
  By Lemma~\ref{lem:AR-identities}, we have $\psi_B(r) = r$ for all $r \in R$.  Thus, any element $f \in \HOM_B \left( {}^{\psi_B}M, B \right)$ is also an element of $\HOM_R(M,B)$, and hence $\tr_B \circ f \in M^\vee$.  The map \eqref{eq:dual-isom} is also clearly of degree $(-\lambda_B,\pi_B)$, since $\tr_B$ is.

  It remains to show that \eqref{eq:dual-isom} is a homomorphism of $(A,B)$-bimodules.  It is clearly a homomorphism of abelian groups.  For $a \in A$ and $f \in \HOM_B \left( {}^{\psi_B}M, B \right)$, we have
  \[
    \tr_B \circ (a \cdot f)
    = (-1)^{\bar{a} \bar{f}} \tr_B \circ f \circ \pr{a}
    = (-1)^{\bar{a} \pi_B} a \cdot (\tr_B \circ f).
  \]
  Thus, \eqref{eq:dual-isom} is a homomorphism of left $A$-modules.  Now let $b \in B$ and $y \in \prescript{\psi_B}{B}M_A$.  Then
  \begin{align*}
    \tr_B \circ (f \cdot b)(y) &= (-1)^{\bar{b}\bar{f}} \tr_B \circ (\pr{b} \circ f)(y) \\
    &= (-1)^{\bar{b}\bar{f}} \tr_B \circ \big( \pr{b}(f(y)) \big) \\
    &= (-1)^{\bar{b}\bar{y}} \tr_B (f(y)b) \\
    &= (-1)^{\bar{b}\bar{f}} \tr_B \big( \psi_B^{-1}(b) f(y) \big) \\
    &= \tr_B(f(by)) \\
    &= \tr_B \circ f \circ \pl{b}(y) \\
    &= \big( (\tr_B \circ f) \cdot b \big)(y)
  \end{align*}
  Thus the map~\eqref{eq:dual-isom} is also a homomorphism of right $B$-modules.
\end{proof}

Let
\[
  \{x_i\}_{i=1}^n \quad \text{and} \quad \{y_i\}_{i=1}^n
\]
be dual sets of generators of $B$ over $R$, where $|x_i| + |y_i| = \lambda_B$ and $\bar x_i + \bar y_i = \pi_B$, for each $i =1,\dotsc,n$ (see Proposition~\ref{prop:dual-sets-of-generators}).

\begin{prop}
  If $M$ is a $(B,A)$-bimodule, then the map
  \begin{equation} \label{eq:dual-isom-inverse}
    M^\vee \to \HOM_B({}^{\psi_B}M, {}_BB_B),\quad
    \theta \mapsto \left( m \mapsto (-1)^{\pi_B(\bar \theta + \bar m)} \sum_{i=1}^n (-1)^{\bar y_i (\pi_B + \bar m)} \theta(y_im) x_i \right),
  \end{equation}
  is a homomorphism of $(A, B)$-bimodules of degree $(\lambda_B, \pi_B)$.  Moreover, the maps \eqref{eq:dual-isom} and \eqref{eq:dual-isom-inverse} are mutually inverse isomorphisms of $(A, B)$-bimodules.
\end{prop}

\begin{proof}
  The map
  \begin{equation} \label{eq:dual-isom-inverse-codomain}
    m \mapsto (-1)^{\pi_B(\bar \theta + \bar m)} \sum_{i=1}^n (-1)^{\bar y_i (\pi_B + \bar m)} \theta(y_im) x_i
  \end{equation}
  is clearly a homomorphism of abelian groups.  Now let $b \in B$ and $m \in {}^{\psi_B}M$.  Then $b \cdot m = \psi_B(b)m$ maps to
  \begin{align*}
    &(-1)^{\pi_B(\bar \theta + \bar b + \bar m)} \sum_{i = 1}^n (-1)^{\bar y_i (\pi_B + \bar b + \bar m)} \theta \big( y_i\psi_B(b)m \big) x_i \\
    &\stackrel{\eqref{eq:dual-bases}}{=} (-1)^{\pi_B(\bar \theta + \bar b + \bar m)} \sum_{i = 1}^n (-1)^{\bar y_i (\pi_B + \bar b + \bar m)} \theta \left(\sum_{j = 1}^n y_j \tr_B \big( x_j y_i \psi_B(b) \big) m \right) x_i \\
    &= (-1)^{\pi_B(\bar \theta + \bar b + \bar m)} \sum_{i, j = 1}^n (-1)^{\bar y_i (\pi_B + \bar b + \bar m) + \bar{y_j}(\pi_B + \bar{x}_j + \bar{y}_i + \bar{b})} \theta \Big( \tr_B \big( x_j y_i \psi_B(b) \big) y_j m \Big) x_i \\
    &= (-1)^{\pi_B(\bar \theta + \bar b + \bar m)} \sum_{i, j = 1}^n (-1)^{\bar y_i (\pi_B + \bar b + \bar m) + (\bar{y}_j + \bar{\theta})(\pi_B + \bar{x}_j + \bar{y}_i + \bar{b})} \tr_B \big( x_j y_i \psi_B(b) \big) \theta(y_j m)x_i \\
    &= (-1)^{\pi_B (\bar \theta + \bar b)} \sum_{i, j = 1}^n (-1)^{\bar y_i (\pi_B + \bar b) + \bar m (\bar x_j + \bar b)} \theta(y_jm)\tr_B(x_j y_i \psi_B(b))x_i \\
    &= (-1)^{\pi_B (\bar \theta + \bar b)} \sum_{i, j = 1}^n (-1)^{\bar y_i \pi_B + \bar m (\bar x_j + \bar b) + \bar b \bar x_j} \theta(y_jm)\tr_B(bx_jy_i)x_i \\
    &\stackrel{\eqref{eq:dual-bases}}{=} (-1)^{\pi_B (\bar \theta + \bar b)} \sum_{j = 1}^n (-1)^{\pi_B \bar y_j + \bar m (\bar x_j + \bar b) + \bar b \bar y_j} \theta(y_jm) b x_j \\
    &= (-1)^{\bar b (\bar \theta + \pi_B) + \pi_B \bar \theta} \sum_{j = 1}^n (-1)^{\pi_B \bar y_j + \bar m \bar x_j} b \theta(y_jm) x_j \\
    &= (-1)^{\bar b (\bar \theta + \pi_B)} b \left( (-1)^{\pi_B(\bar \theta + \bar m)} \sum_{j = 1}^n (-1)^{\bar y_j (\pi_B + \bar m)} \theta(y_jm) x_j \right).
  \end{align*}
  Thus \eqref{eq:dual-isom-inverse-codomain} is a homomorphism of left $B$-modules of degree $(\lambda_B,\pi_B)$.  Since the (set-theoretic) inverse of a bimodule homomorphism is also a bimodule homomorphism, it remains to show that \eqref{eq:dual-isom} and \eqref{eq:dual-isom-inverse} are mutually inverse.

  Let $f \in \HOM_B({}^{\psi_B}M, {}_BB_B)$. The map~\eqref{eq:dual-isom} followed by the map~\eqref{eq:dual-isom-inverse} sends $f$ to the map
  \begin{align*}
    m
    &\mapsto (-1)^{\pi_B(\pi_B + \bar f + \bar m)} \sum_{i=1}^n (-1)^{\bar y_i (\pi_B + \bar m)} \tr_B \big( f(y_im) \big) x_i \\
    &= (-1)^{\pi_B(\pi_B + \bar f + \bar m)} \sum_{i=1}^n (-1)^{\bar y_i (\pi_B + \bar m + \bar f)} \tr_B \big( \psi_B^{-1}(y_i) f(m) \big) x_i \\
    &= (-1)^{\pi_B(\bar f + \bar m)} \sum_{i=1}^n (-1)^{\pi_B \bar x_i} \tr_B \big( f(m) y_i \big) x_i \\
    &\stackrel{\eqref{eq:dual-bases}}{=} f(m).
  \end{align*}
  Thus \eqref{eq:dual-isom-inverse} is left inverse to \eqref{eq:dual-isom}.

  Now let $\theta \in M^\vee$.  The map \eqref{eq:dual-isom-inverse} followed by the map \eqref{eq:dual-isom} sends $\theta$ to the map
  \begin{align*}
    m
    &\mapsto (-1)^{\pi_B(\bar \theta + \bar m)} \sum_{i=1}^n (-1)^{\bar y_i (\pi_B + \bar m)} \tr_B \big( \theta(y_im) x_i \big) \\
    &= \sum_{i=1}^n (-1)^{\bar y_i \bar m} \theta(y_i m) \tr_B (x_i) \\
    &= \sum_{i=1}^n (-1)^{\bar y_i (\bar \theta + \bar y_i)} \tr_B (x_i) \theta(y_im) \\
    &= \sum_{i=1}^n (-1)^{\bar y_i} \theta \big( \tr_B (x_i) y_i m \big) \\
    &= \sum_{i=1}^n \theta \big( y_i \tr_B (x_i) m \big) \\
    &= \theta \left( \sum_{i=1}^n y_i \tr_B (x_i) m \right) \\
    &\stackrel{\eqref{eq:dual-bases}}{=} \theta(m).
  \end{align*}
  Hence \eqref{eq:dual-isom-inverse} is also right inverse to \eqref{eq:dual-isom}.
\end{proof}

We will let
\[
  \kappa \colon \left( {}_BA^{\psi_A}_A \right)^\vee \xrightarrow{\cong} \HOM_B \left( \prescript{\psi_B}{B}A^{\psi_A}_A, {}_BB_B \right)
\]
be the special case of the isomorphism~\eqref{eq:dual-isom-inverse} of $(A, B)$-bimodules where one takes $M$ to be ${}_BA^{\psi_A}_A$.

\begin{prop}
  The map
  \[
    \varphi_A \colon {}_A A_B \to \left( \prescript{}{B} A^{\psi_A}_A \right)^\vee, \quad \varphi_A(a) = \tr_A \circ \pr{\psi_A(a)},
  \]
  is an isomorphism of $(A, B)$-bimodules of degree $(-\lambda_A, \pi_A)$.
\end{prop}

\begin{proof}
  The map $\varphi_A$ is clearly a homomorphism of abelian groups.  Let $r \in R$, $a \in A$, and $x \in \prescript{}{A}A^{\psi_A}_B$.  Then
  \begin{multline*}
    \varphi_A(a)(rx)
    = \tr_A \circ \pr{\psi_A(a)}(rx) = (-1)^{\bar{a}(\bar{r} + \bar{x})} \tr_A \big( rx\psi_A(a) \big) \\
    = (-1)^{\bar{a}(\bar{r} + \bar{x}) + \pi_A \bar{r}} r\tr_A(x\psi_A(a))
    = (-1)^{\bar{r}(\bar{a} + \pi_A)} r\tr_A \circ \pr{\psi_A(a)}(x)
    = (-1)^{\bar{r}(\bar{a} + \pi_A)}r\varphi_A(a)(x).
  \end{multline*}
  Thus, $\varphi_A(a) \in \left( \prescript{}{B} A^{\psi_A}_A \right)^\vee$.

  Now, for $a, a' \in A$ and $x \in \prescript{}{A}A^{\psi_A}_B$, we have
  \begin{align*}
    \varphi_A(a'a)(x) &= \tr_A \circ \pr{\psi_A(a'a)}(x) \\
    &= (-1)^{\bar{x}(\bar{a}' + \bar{a})} \tr_A(x\psi_A(a'a)) \\
    &= (-1)^{\bar{x}(\bar{a}' + \bar{a})} \tr_A(x\psi_A(a')\psi_A(a)) \\
    &= (-1)^{\bar{a}'(\bar x + \bar a)} \tr_A \circ \pr{\psi_A(a)}(x \psi_A(a')) \\
    &= (-1)^{\bar{a}'(\bar{x} + \bar{a})} \varphi_A(a)(x \psi_A(a')) \\
    &= (-1)^{\bar{a}' \bar{a}} \varphi_A(a) \circ \pr{\psi_A(a')}(x) \\
    &= (-1)^{\bar{a}' \pi_A} \big( a' \cdot \varphi_A(a) \big)(x).
  \end{align*}
  Thus $\varphi_A$ is a homomorphism of left $A$-modules of degree $(-\lambda_A,\pi_A)$.

  On the other hand, for $a \in A$, $b \in B$, and $x \in \prescript{}{A}A^{\psi_A}_B$, we have
  \begin{align*}
    \varphi_A(ab)(x)
    &= \tr_A \circ \pr{\psi_A(ab)}(x) \\
    &= (-1)^{(\bar{a} + \bar{b})\bar{x}} \tr_A \big( x\psi_A(ab) \big) \\
    &= (-1)^{(\bar{a} + \bar{b})\bar{x}} \tr_A \big( x\psi_A(a)\psi_A(b) \big) \\
    &= (-1)^{\bar{a} (\bar{x} + \bar{b})} \tr_A \big( bx\psi_A(a) \big) \\
    &= \tr_A \circ \pr{\psi_A(a)}(bx) \\
    &= \tr_A \circ \pr{\psi_A(a)} \circ \pl{b}(x) \\
    &= (\varphi_A(a) \cdot b)(x).
  \end{align*}
  Thus $\varphi_A$ is a homomorphism of right $B$-modules.

  It remains to show that $\varphi_A$ is an isomorphism.  Suppose $\varphi(a) = \varphi(a')$ for some $a,a' \in A$.  This implies $\bar a = \bar a'$.  Then, for all $x \in \prescript{}{A}A^{\psi_A}_A$, we have
  \begin{gather*}
    \varphi(a)(x) = \varphi(a')(x) \\
    \implies \tr_A \circ \pr{a}(x) = \tr_A \circ \pr{a'}(x) \\
    \implies (-1)^{\bar{a}\bar{x}} \tr_A \big( x\psi_A(a) \big) = (-1)^{\bar{a}'\bar{x}} \tr_A \big( x\psi_A(a') \big) \\
    \implies 0 = (-1)^{\bar{a}\bar{x}} \tr_A \Big( x \big( \psi_A(a) - \psi_A(a') \big) \Big).
  \end{gather*}
  It thus follows from \ref{T1} that $\psi_A(a) = \psi_A(a')$, and hence $a = a'$.  Thus $\varphi_A$ is injective.

  Now, every element $\varphi \in \left( \prescript{}{B} A^{\psi_A}_A \right)^\vee$ can be viewed as an element of $\HOM_R({}_R A, R)$.  Then, by \ref{T2}, there exists an $a \in A$ such that $\varphi = \tr_A \circ \pr{a}$.  Since $\psi_A$ is a ring isomorphism, we have
  \[
    \tr_A \circ \pr{\psi_A(\psi_A^{-1}(a))} = \tr_A \circ \pr{a} = \varphi.
  \]
  Thus, $\varphi_A$ is surjective.
\end{proof}

\begin{prop} \label{prop:frobenius-extension-L2}
  The map
  \[
    \kappa \circ \varphi_A \colon {}_A A_B \to \HOM_B \left( \prescript{\psi_B}{B}A^{\psi_A}_A, \prescript{}{B}B_B \right)
  \]
  is an isomorphism of $(A, B)$-bimodules of degree $(\lambda_B - \lambda_A, \pi_A + \pi_B)$.
  Moreover, the map
  \[
    \tr \colon \prescript{\psi_B}{B}A^{\psi_A}_B \to \prescript{}{B}B_B, \quad \tr(a) = (-1)^{\pi_B (\pi_A + \bar a)} \sum_{i = 1}^n (-1)^{\bar y_i (\pi_B + \bar a)} \tr_A(y_i  a)x_i
  \]
  is a trace map, i.e., it satisfies conditions \ref{T1} and \ref{T2}.
\end{prop}

\begin{proof}
  Since $\kappa \circ \varphi_A$ is a composition of $(A, B)$-bimodule isomorphisms, it too is an $(A, B)$-bimodule isomorphism.  Now, for $a \in \prescript{}{A}A_B$, we have
  \begin{align*}
    (\kappa \circ \varphi_A)(1_A)(a)
    &= \Big( \kappa \big( \tr_A \circ \pr{\psi_A(1_A)} \big) \Big)(a) \\
    &= \big( \kappa (\tr_A) \big) (a) \\
    &= (-1)^{\pi_B (\pi_A + \bar a)} \sum_{i = 1}^n (-1)^{\bar y_i (\pi_B + \bar a)} \tr_A(y_i  a)x_i.
  \end{align*}
  Then by \cite[Prop.~4.1]{PS16}, $\tr$ is left trace map.
\end{proof}

\begin{theo} \label{theo:main-theorem}
  Let $A$ be a ring extension of $B$, and $B$ be a ring extension of $R$, with $C_A(R) = A$.  Suppose that $A$ is a Frobenius extension of $R$ of degree $(-\lambda_A, \pi_A)$, with Nakayama automorphism $\psi_A$, and that $B$ is a Frobenius extension of $R$ of degree $(-\lambda_B, \pi_B)$, with Nakayama automorphism $\psi_B$.   If $A$ is projective as a left $B$-module, then $A$ is a $(\psi_A, \psi_B)$-Frobenius extension of $B$ of degree $(\lambda_B - \lambda_A, \pi_B + \pi_A)$.  Moreover, the induction functor ${}_A A_B \otimes_B -$ is right adjoint to the shifted twisted restriction functor $\{\lambda_B - \lambda_A, \pi_B + \pi_A\} \prescript{\psi_B}{B}A^{\psi_A}_A \otimes_A -$.
\end{theo}

\begin{proof}
  Since $A$ is a Frobenius extension of $R$, it is finitely generated as an $R$-module, and hence also finitely generated as a left $B$-module.  Moreover, by Proposition~\ref{prop:frobenius-extension-L2}, there is a trace map satisfying \ref{T1} and \ref{T2}.  Thus $A$ is an $(\psi_A, \psi_B)$-Frobenius extension of $B$.  The final assertion follows from \cite[Th.~6.2]{PS16}.
\end{proof}

\begin{rem}
  Recall that, by Corollary~\ref{cor:R-central-implies-untwisted}, we would gain no generality in Theorem~\ref{theo:main-theorem} by allowing for $A$ and $B$ to be \emph{twisted} Frobenius extensions of $R$.  In the case that $R$ is a field, concentrated in degree $(0,0)$, Theorem~\ref{theo:main-theorem} recovers \cite[Cor.~7.4]{PS16}.
\end{rem}

%
\section{Applications} \label{sec:applications}
%

In this final section, we give several examples that illustrate Theorem~\ref{theo:main-theorem}.  In particular, we see that a number of results that have appeared in the literature follow immediately from this theorem.

\begin{eg}[Group rings]
  Let $R$ be a supercommutative ring, $G$ a finite group, and $H$ a subgroup of $G$.  Consider the following group rings over $R$:
  \[
    R \cong R[\{e\}] \subseteq R[H] \subseteq R[G],
  \]
  where $e$ is the identity element of $G$.  By construction, $R[H]$ and $R[G]$ are free as $R$-modules.  It is easy to verify that the map
  \[
    \tr \colon R[G] \to R,\quad \tr \left( \sum_{g \in G} r_g g \right) = r_e,
  \]
  satisfies \ref{T1} and \ref{T2} with $\alpha$ and $\beta$ both the identity map.  Thus $R[G]$ and $R[H]$ are both untwisted Frobenius extensions of $R$ and their Nakayama automorphisms are the identity automorphisms.  The ring $R$ clearly lies in the center of $R[G]$ and $R[G]$ is free as a left $R[H]$-module, with basis given by a set of left coset representatives.  Therefore, by Theorem~\ref{theo:main-theorem}, $R[G]$ is an untwisted Frobenius extension of $R[H]$.  In the case that $R$ is concentrated in degree $(0,0)$, this recovers the well-known result that a finite group ring is a Frobenius extension of a subgroup ring.
\end{eg}

\begin{eg}[Hopf algebras]
  Let $R$ be an unique factorization domain, let $H$ be a Hopf algebra over $R$ that is finitely generated and projective as an $R$-module, and let $K$ be a Hopf subalgebra of $H$.  Then $H$ and $K$ are both untwisted Frobenius extensions of $R$ by \cite[Cor.~1]{Par71}.  Let $\psi_H$ and $\psi_K$ denote their respective Nakayama automorphisms.  If $H$ is projective as a left $K$-module (this condition is automatically satisfied when $R$ is a field by \cite[Th.~7]{NZ89}), then $H$ is a $(\psi_H, \psi_K)$-Frobenius extension of $K$ by Theorem~\ref{theo:main-theorem}.  Moreover, we have that $H$ is an $(\id_H, \psi_K \circ \psi_H^{-1})$-Frobenius extension of $K$ by applying \cite[Proposition~3.4]{PS16} with $u = 1_H$.  That is, it is a Frobenius extension of the second kind.  Thus we recover the result \cite[Cor.~3.6(1)]{Sch92}.
\end{eg}

\begin{eg}[Nilcoxeter rings]
  Let $R$ be a supercommutative ring and fix a nonnegative integer $n$.  The nilcoxeter ring $N_n$ over $R$ is generated by the elements $u_1, \cdots ,u_{n - 1}$ with the relations
  \begin{gather*}
    u_i^2 = 0\text{ for } 1 \leq i \leq n - 1, \\
    u_i u _j = u_j u_i \text{ for } 1 \leq i, j \leq n - 1 \text{ such that } |i - j| > 1, \\
    u_i u_{i + 1} u_i = u_{i + 1} u_i u_{i + 1} \text{ for } 1 \leq i < n - 1.
  \end{gather*}
  As an $R$-module, $N_n$ has the basis $\{u_w \mid w \in S_n\}$, where $S_n$ is the symmetric group on $n$ elements.  Multiplication of basis elements is given by
  \[
    u_vu_w =
    \begin{cases}
      u_{vw} & \text{if } \ell(v + w) = \ell(v) + \ell(w), \\
      0 & \text{if } \ell(v + w) \neq \ell(v) + \ell(w),
    \end{cases}
  \]
  where $\ell$ is the length function of the symmetric group.  So $N_n$ is free and thus projective as an $R$-module.  Now consider the $R$-linear function determined by
  \[
    \tr_n \colon N_n \to R,\quad
    \tr_n(u_w) =
    \begin{cases}
      1 & \text{if } w = w_0 \in S_n, \\
      0 & \text{if } w \neq w_0 \in S_n,
    \end{cases}
  \]
  where $w_0$ denotes the permutation of maximal length in $S_n$.  It can be shown that $N_n$ is an untwisted Frobenius extension of $R$ of degree $(- {n \choose 2}, {n \choose 2})$ with trace map $\tr_n$, and the Nakayama automorphism associated to $\tr_n$ is given by $\psi_n(u_i) = u_{n - i}$; see \cite[Lem~8.2]{PS16}, where one replaces $\F$ with $R$.  Although the author of \cite[Prop~4]{Kho01} works over the field $\Q$, his proof that $N_n$ is projective as a left $N_{n - 1}$-module still holds over $R$.  It is clear that $C_{N_n}(R) = N_n$.  Therefore, by Theorem~\ref{theo:main-theorem}, $N_n$ is a $(\psi_{n}, \psi_{n - 1})$-Frobenius extension of $N_{n - 1}$ of degree $\left( {n - 1 \choose 2} - {n \choose 2}, {n \choose 2} + {n - 1 \choose 2} \right)$.
\end{eg}


\bibliographystyle{alphaurl}
\bibliography{PoonSavage-NestedFrobeniusExtensions}

\end{document}